\documentclass{amsart}
\usepackage{setspace, amsmath, amsthm, amssymb, amsfonts, amscd, epic, graphicx, ulem, dsfont}
\usepackage[T1]{fontenc}
\usepackage{multirow}
\usepackage{bbm}
\usepackage{enumerate}

\makeatletter \@namedef{subjclassname@2010}{
  \textup{2010} Mathematics Subject Classification}
\makeatother

\newtheorem{thm}{Theorem}[section]

\newtheorem{lem}[thm]{Lemma}
\newtheorem{pro}[thm]{Proposition}
\newtheorem{conj}[thm]{Conjecture}

\theoremstyle{remark}

\theoremstyle{definition}

\newcommand{\R}{\mathbb{R}}

\begin{document}

\title[Commutators of Unbounded Operators]{Counterexamples Related to Commutators of Unbounded Operators}
\author[M. H. MORTAD]{Mohammed Hichem Mortad}

\dedicatory{}
\thanks{}
\date{}


 \address{Department of
Mathematics, University of Oran 1, Ahmed Ben Bella, B.P. 1524, El
Menouar, Oran 31000, Algeria.\newline {\bf Preferred mailing
address}:
\newline Pr Mohammed Hichem Mortad \newline BP 7085 Seddikia Oran
\newline 31013 \newline Algeria}

\email{mhmortad@gmail.com, mortad.hichem@univ-oran1.dz.}

\begin{abstract}
The present paper is exclusively devoted to counterexamples about
commutators and self commutators of unbounded operators on a Hilbert
space. As a bonus, we provide a simpler counterexample than
McIntosh's famous example obtained some while ago.
\end{abstract}

\maketitle

\section{Introduction}

The formal commutator of two non necessarily bounded operators $A$
and $B$ is defined to be $AB-BA$. We have called it "formal" as,
unlike the bounded case, $AB-BA=0$ does not always imply the
(strong) commutativity of $A$ and $B$ when say $A$ and $B$ are
self-adjoint. The first such counterexample is due to Nelson
\cite{Nelson-Analytic-vectors}. Apparently, the first textbook to
include it is \cite{RS1}. The same example is developed in detail in
\cite{Schmudgen-Operator-algebra}, pp 257-258. Perhaps the simplest
counterexample is due to Schm\"{u}dgen and may be found in e.g.
\cite{SCHMUDG-book-2012}.

The self commutator of a densely defined operator $T$ is defined to
be $TT^*-T^*T$.

The main purpose of this paper is to exhibit counterexamples to
questions related to commutators and self-commutators as regards
boundedness, closedness and self-adjointness.

Readers throughout the paper will observe how fascinating the use of
matrices of unbounded operators helps to find such counterexamples.
The same approach has equally allowed us to find more interesting
counterexamples on a different topic. See
\cite{Mortad-Powers-Trivial-domains-CEXAMP}. Readers should be wary
that it took me some time before I came up with these relatively
simple examples.

We refer readers to \cite{tretetr-book-BLOCK} for properties and
results about matrices of unbounded operators. See also
\cite{Moller-Szafanriac-matri-unbounded} or
\cite{Ota-Schmudgen-2003-Matrices-Operators}. For the general theory
of unbounded operators, readers may wish to consult
\cite{SCHMUDG-book-2012} or \cite{Weidmann}. See also
\cite{Putnam-book-1967}. Finally, it is worth noticing that there is
an extensive work on estimating the norm of commutators of some
classes of bounded operators. We cite among others:
\cite{Fong-Self-Commut-1986}, \cite{kittaneh-2007-Commutator-JFA},
\cite{kittaneh-2008-norm-commuta-MZ} and
\cite{kittaneh-2008-COMMUT-IEOT}.

\section{Main Counterexamples}

We start with an auxiliary result which is also interesting in its
own.

\begin{pro}\label{B unbounded B2 unbounded abs value QU} There exists
a densely defined unbounded and closed operator $B$ such that $B^2$
and $|B|B$ are bounded whereas $B|B|$ is unbounded and closed.
\end{pro}

\begin{proof}Let $H$ be a complex Hilbert space and let $A$ be an unbounded self-adjoint and positive operator with domain
$D(A)\subset H$ (for instance $Af(x)=e^{x^2}f(x)$ on its maximal
domain on $L^2(\R)$). Let
\[B=\left(
      \begin{array}{cc}
        0 & A \\
        0 & 0 \\
      \end{array}
    \right)
\]
be defined on $H\oplus D(A)$. Then $B$ is closed and as seen before
$B^2=0$ on $H\oplus D(A)$. Now,
\[|B|=\left(
        \begin{array}{cc}
          0 & 0 \\
          0 & A \\
        \end{array}
      \right)\text{ and so }|B|B=\left(
                                   \begin{array}{cc}
                                     0 & 0_{D(A)} \\
                                     0 & 0_{D(A)} \\
                                   \end{array}
                                 \right),
\]
that is, $|B|B$ is bounded on $H\oplus D(A)$. However,
\[B|B|=\left(
      \begin{array}{cc}
        0 & A^2 \\
        0 & 0 \\
      \end{array}
    \right)\]
is clearly unbounded on $H\oplus D(A^2)$.

\end{proof}

\begin{pro}There are two densely defined unbounded and closed
operators $B$ and $C$ such that $CB-BC$ is bounded and unclosed
while $|C|B-B|C|$ is unbounded and closed.
\end{pro}

\begin{proof}Let $B$ be as in Proposition \ref{B unbounded B2 unbounded abs value
QU} and set $C=B$. Then $CB-BC=0_{D(B^2)}$ is clearly bounded and
unclosed. By a glance at Proposition \ref{B unbounded B2 unbounded
abs value QU} again, we easily see that
\[|B|B-B|B|=\left(
              \begin{array}{cc}
                0 & A^2 \\
                0 & 0 \\
              \end{array}
            \right)
\]
and that it is closed (and unbounded) on $D(|B|B-B|B|)=H\oplus
D(A^2)$.

\end{proof}

\begin{lem}
There are two unbounded and self-adjoint operators $A$ and $B$ such
that $D(A)=D(B)$, $A^2-B^2$ is bounded but $AB-BA$ is unbounded.
\end{lem}

\begin{proof}
Let $T$ be any unbounded and self-adjoint operator with domain
$D(T)\subset H$ where $H$ is a Hilbert space. Next, define $A=\left(
                                                                   \begin{array}{cc}
                                                                     0 & T \\
                                                                     T & 0 \\
                                                                   \end{array}
                                                                 \right)$
                                                                 and
                                                                 $B=\left(
                                                                      \begin{array}{cc}
                                                                        T & 0 \\
                                                                        0 & -T \\
                                                                      \end{array}
                                                                    \right)$
                                                                    and
                                                                    so
                                                                    $D(A)=D(B)=D(T)\oplus
                                                                    D(T)$.
                                                                    Hence
                                                                    $A^2=B^2=\left(
                                                                      \begin{array}{cc}
                                                                        T^2 & 0 \\
                                                                        0 & T^2 \\
                                                                      \end{array}
                                                                    \right)$
                                                                    and
                                                                    so
                                                                    $A^2-B^2$
                                                                    is
                                                                    bounded
                                                                    (on
                                                                    $D(A^2)$)
                                                                    whilst
\[AB-BA=\left(
                                                                   \begin{array}{cc}
                                                                     0 & -2T^2 \\
                                                                     2T^2 & 0 \\
                                                                   \end{array}
                                                                 \right)\]
is obviously unbounded.
\end{proof}

We know that there exist two unbounded self-adjoint operators $A$
and $B$ such that $AB-BA$ is bounded (on its domain) while
$|A|B-B|A|$ is unbounded. The first (and apparently the only)
counterexample is due to McIntosh in
\cite{McIntosh-commutators-counterexample} who answered a question
raised by the great T. Kato. The example we are about to give here
is new and simpler than McIntosh's. Moreover, in our case both
$AB-BA$ and $|A|B-B|A|$ are even closed.

\begin{pro}\label{MCIntosh-NEw-EXAMPLE-MORTAD}
There exist two unbounded and self-adjoint operators $A$ and $B$
such that $AB-BA$ is bounded (and closed) whilst $|A|B-B|A|$ is
unbounded (and also closed).
\end{pro}

The counterexample is based on the following recently obtained
result:

\begin{lem}\label{mmm}(\cite{Dehimi-Mortad-Chernoff}) There are unbounded self-adjoint operators $A$ and $B$
such that
\[D(A^{-1}B)=D(BA^{-1})=\{0\}\]
(where $A^{-1}$ and $B^{-1}$ are not bounded).
\end{lem}

Now, we give the promised counterexample:

\begin{proof}

Let $R,S,T$ be three self-adjoint operators on a Hilbert space $H$
with domains $D(R)$, $D(S)$ and $D(T)$ respectively. Assume also
that $S$ is positive. Now, define on $H\oplus H$ the operators
\[A=\left(
      \begin{array}{cc}
        0 & S \\
        S & 0 \\
      \end{array}
    \right)\text{ and }
B=\left(
    \begin{array}{cc}
      T & 0 \\
      0 & R \\
    \end{array}
  \right)
\]
with domains $D(A)=D(S)\oplus D(S)$ and $D(B)=D(T)\oplus D(R)$
respectively. Hence
\[AB-BA=\left(
      \begin{array}{cc}
        0 & SR-TS \\
        ST-RS & 0 \\
      \end{array}
    \right).\]
    Since $|A|=\left(
    \begin{array}{cc}
      S & 0 \\
      0 & S \\
    \end{array}
  \right)$, it follows that
  \[|A|B-B|A|=\left(
    \begin{array}{cc}
      ST-TS & 0 \\
      0 & SR-RS \\
    \end{array}
  \right).\]

  To obtain the appropriate operators, let $C$ and $D$ be such that $D(CD)=D(DC)=\{0\}$ (as in
  Lemma \ref{mmm}). Remember that $C$ is self-adjoint, positive, unbounded and (not boundedly) invertible . Define now
  \[S=\left(
        \begin{array}{cc}
          0 & C \\
          C & 0 \\
        \end{array}
      \right)\text{ and }T=\left(
                             \begin{array}{cc}
                               D & 0 \\
                               0 & D \\
                             \end{array}
                           \right)
  \]
  and so
  \[ST=\left(
         \begin{array}{cc}
           0 & CD \\
           CD & 0 \\
         \end{array}
       \right)\text{ and }TS=\left(
         \begin{array}{cc}
           0 & DC \\
           DC & 0 \\
         \end{array}
       \right).
  \]
  Hence $D(ST)=D(TS)=\{0_{[L^2(\R)]^2}\}$. This says that $AB-BA$ is
  bounded on $L^2(\R)\oplus L^2(\R)\oplus L^2(\R)\oplus L^2(\R)$. In
  fact, $AB-BA$ is trivially bounded as it is only defined on
  $\{0\}$ and $AB-BA$ is therefore closed. In order that $|A|B-B|A|$ be
  unbounded, it suffices then to exhibit a self-adjoint and unbounded $R$ such
  that $SR-RS$ is unbounded. Consider $R=\left(
                                           \begin{array}{cc}
                                             C^{-\frac{1}{2}} & 0 \\
                                             0 & C \\
                                           \end{array}
                                         \right)$. Therefore,
  \[SR=\left(
         \begin{array}{cc}
           0 & C^2 \\
           \sqrt C & 0 \\
         \end{array}
       \right),~RS=\left(
         \begin{array}{cc}
           0 & \sqrt C \\
           C^2 & 0 \\
         \end{array}
       \right)\]
       \text{ and }
       \[SR-RS=\left(
         \begin{array}{cc}
           0 & C^2-\sqrt C \\
           \sqrt C-C^2 & 0 \\
         \end{array}
       \right).
  \]

Finally, $C^2-\sqrt C $ is unbounded (and self-adjoint). Indeed,
since $C$ is self-adjoint, it is unitarily equivalent to the
multiplication operator $M_{\varphi}$ by an unbounded and
real-valued (positive) function $\varphi$. Hence, $C^2-\sqrt C$ is
unitarily equivalent to the multiplication operator by the unbounded
real-valued $\varphi^2-\sqrt{\varphi}$. Thus, as $C^2-\sqrt C $ is
unbounded, $SR-RS$ too is unbounded and so  $|A|B-B|A|$ is equally
  unbounded (and closed), as coveted.
\end{proof}

Now, we consider the case of self-commutators. First, we show that
the self-commutator of a densely defined and closed operator may
only be defined at 0.

\begin{pro}There exists a densely defined and closed operator $T$ such that $D(TT^*)\cap D(TT^*)=\{0\}$ and hence
\[D(TT^*-T^*T)=\{0\}.\]
\end{pro}

\begin{proof}Consider two unbounded and self-adjoint operators $A$ and $B$
which obey $D(A)\cap D(B)=\{0\}$ (see e.g. \cite{KOS}). Set
$T=\left(
                                                                   \begin{array}{cc}
                                                                     0 & A \\
                                                                     B & 0 \\
                                                                   \end{array}
                                                                 \right)$ and so $T^*=\left(
                                                                   \begin{array}{cc}
                                                                     0 & B \\
                                                                     A & 0 \\
                                                                   \end{array}
                                                                 \right)$.
Hence
\[TT^*=\left(
              \begin{array}{cc}
                A^2 & 0 \\
                0 & B^2 \\
              \end{array}
            \right) \text{ and }T^*T=\left(
              \begin{array}{cc}
                B^2 & 0 \\
                0 & A^2 \\
              \end{array}
            \right).
\]
Therefore,
\[D(TT^*)\cap D(TT^*)=[D(A^2)\cap D(B^2)]\oplus [D(B^2)\cap D(A^2)]\]
and so
\[D(TT^*)\cap D(TT^*)\subset[D(A)\cap D(B)]\oplus [D(B)\cap D(A)]=\{(0,0)\},\]
as needed.
\end{proof}

In the next two counterexamples, we show that $TT^*-T^*T$ may be
bounded but $|T||T^*|-|T^*||T|$ may be not, and vice versa.

\begin{pro}
There exists a densely defined and closed operator $T$ such that
$TT^*-T^*T$ is unbounded but $|T||T^*|-|T^*||T|$ is bounded (even
zero on some domain!).
\end{pro}

\begin{proof}
Let $A$ be an unbounded, self-adjoint and positive operator with
domain $D(A)\subset H$. Define $T$ with domain $D(T)=D(A)\oplus H$
by
\[T=\left(
      \begin{array}{cc}
        0 & 0 \\
        A & 0 \\
      \end{array}
    \right)\text{ and so }T^*=\left(
      \begin{array}{cc}
        0 & A \\
        0 & 0 \\
      \end{array}
    \right).
\]
Hence
\[TT^*=\left(
         \begin{array}{cc}
           0 & 0 \\
           0 & A^2 \\
         \end{array}
       \right)\text{ and }
T^*T=\left(
         \begin{array}{cc}
           A^2 & 0 \\
           0 & 0 \\
         \end{array}
       \right).\]
Also,
\[|T||T^*|=\left(
             \begin{array}{cc}
              0  & 0_{D(A)} \\
               0 & 0_{D(A)} \\
             \end{array}
           \right)
\text{ and }|T^*||T|=\left(
             \begin{array}{cc}
               0_{D(A)} & 0 \\
              0_{D(A)}  & 0 \\
             \end{array}
           \right).\]
Thus,
\[TT^*-T^*T=\left(
         \begin{array}{cc}
           -A^2 & 0 \\
           0 & A^2 \\
         \end{array}
       \right)\]
is clearly unbounded (and self-adjoint) whereas $|T||T^*|-|T^*||T|$
is bounded (in fact, it is the zero operator on $D(A)\oplus D(A)$).
\end{proof}

\begin{pro}
There exists a densely defined and closed operator $T$ such that
$TT^*-T^*T$ is bounded whereas $|T||T^*|-|T^*||T|$ is unbounded.
\end{pro}

\begin{proof}
Let $A$ and $B$ be two self-adjoint and positive operators with
domains $D(A)$ and $D(B)$. Set $T=\left(
                            \begin{array}{cc}
                              0 & \sqrt{A} \\
                              \sqrt{B} & 0 \\
                            \end{array}
                          \right)$. Hence
                          \[TT^*-T^*T=\left(
      \begin{array}{cc}
        A-B & 0 \\
        0 & B-A \\
      \end{array}
    \right)\]
and
\[|T||T^*|-|T^*||T|=\left(
        \begin{array}{cc}
          \sqrt{B}\sqrt{A}-\sqrt{A}\sqrt{B} & 0 \\
          0 & \sqrt{A}\sqrt{B}-\sqrt{B}\sqrt{A} \\
        \end{array}
      \right).\]

To get the desired counterexample, it suffices to have $D(A)\cap
D(B)=\{0\}$ and so $TT^*-T^*T$ becomes trivially bounded, and at the
same time, the same pair must make
$\sqrt{A}\sqrt{B}-\sqrt{B}\sqrt{A}$ unbounded (and thus
$|T||T^*|-|T^*||T|$ too becomes unbounded). One possible choice is
to consider
 $\sqrt A$
to be defined by
\[\sqrt Af(x)=e^{\frac{x^2}{4}}f(x)\]
on $D(\sqrt A)=\{f\in L^2(\R):e^{\frac{x^2}{4}}f\in L^2(\R)\}$. Then
$\sqrt A$ is self-adjoint, positive and boundedly invertible. Set
$\sqrt B:=\mathcal{F}^*\sqrt A\mathcal{F}$ where $\mathcal{F}$
denotes the usual $L^2(\R)$-Fourier Transform. Then $\sqrt B$ too is
self-adjoint, positive and boundedly invertible. We know that
$D(A)\cap D(B)=\{0\}$ (see e.g. \cite{KOS}). So, it only remains to
check that $\sqrt{B}\sqrt{A}-\sqrt{A}\sqrt{B}$ is in unbounded. A
way of seeing that is to consider the action of this operator on
some sequence of Gaussian functions. This finishes the proof.
\end{proof}

\section*{Conjecture}

Inspired by the pair $S$ and $R$ which appeared in Proposition
\ref{MCIntosh-NEw-EXAMPLE-MORTAD}, we propose the following
conjecture:

\begin{conj}
For any unbounded self-adjoint operator $A$, there is always a
bounded and self-adjoint operator $B$ such that $AB-BA$ is
unbounded.
\end{conj}

\end{document}